\documentclass[11pt]{article}

\usepackage{mathrsfs,amsmath,amssymb,bm,subfigure,float,geometry}
\usepackage{booktabs,mathtools,algorithm}
\usepackage{lipsum}
\usepackage{amsfonts}
\usepackage{graphicx}
\usepackage{epstopdf}
\usepackage{makecell, rotating}
\usepackage{multirow}
\usepackage{algorithmic}
\usepackage{ntheorem}
\usepackage{tikz}

\usepackage{indentfirst}
\usepackage[ulem={normalem,normalbf}]{changes}

\newcommand{\bx}{\bm{x}}

\newcommand{\bp}{\bm{p}}

\newcommand{\bu}{\bm{u}}

\newcommand{\bbR}{\mathbb{R}}
\newcommand{\bbC}{\mathbb{C}}

\newcommand{\calR}{\mathcal{R}}

\newtheorem*{proof}{Proof}
\newtheorem{thm}{Theorem}[section]

\newtheorem{proposition}[thm]{Proposition}

\newtheorem{remark}[thm]{Remark}
\newtheorem{example}[thm]{Example}

\allowdisplaybreaks[4]

\begin{document}
\renewcommand\arraystretch{1}

\title{A general alternating-direction implicit Newton method for solving complex continuous-time algebraic Riccati matrix equation
\thanks{The work was supported in part by the National Natural Science Foundation
of China (12171412, 11771370), Natural Science Foundation for Distinguished Young Scholars of
Hunan Province (2021JJ10037), Hunan Youth Science and Technology Innovation Talents Project
(2021RC3110), the Key Project of Education Department of Hunan Province (19A500, 21A0116).}}

\author{ Shifeng Li, Kai Jiang~~Juan Zhang\thanks{Corresponding author(zhangjuan@xtu.edu.cn).} \thanks{Key Laboratory of Intelligent Computing and Information Processing of Ministry of Education, Hunan Key Laboratory for Computation and Simulation in Science and Engineering, School of Mathematics and Computational
		Science, Xiangtan University, Xiangtan, Hunan,
		411105, P. R. China.}}

\date{}
\maketitle
\setlength{\parindent}{2em}
\textbf{Abstract}
In this paper, applying the Newton method, we transform the complex continuous-time algebraic Riccati matrix equation into a Lyapunov equation. Then, we introduce an efficient general alternating-direction implicit (GADI) method to solve the Lyapunov equation. The inexact Newton-GADI method is presented to save computational amount effectively. Moreover, we analyze the convergence of the Newton-GADI method. The convergence rate of the Newton-GADI and Newton-ADI methods is compared  by analyzing
their spectral radii. Furthermore, we give a way to select the quasi-optimal parameter.
Corresponding numerical tests are shown to illustrate the effectiveness of the proposed algorithms.


\section{Introduction}

In this paper, we consider the complex continuous-time algebraic
Riccati matrix equation (CARE) defined by
\begin{align}
A^*X+XA-XKX+Q=0,
\label{eq:CARE}
\end{align}
where $A,~K,~Q\in\bbC^{n\times n}$,
$K=K^*$, $Q=Q^*$, and $X$ is an unknown matrix.

The CARE arises in many areas such as linear and non-linear optimal control systems \cite{P.L.1995Algebraic}-\cite{G.1990Worst},
 Wiener-Hopf factorization of Markov chains \cite{D.1982potential},
 total least squares problems \cite{B.J.1992total}, spectral factorization of rational matrix functions \cite{K.I.1981factorization}-\cite{I.M.1987inverse},
 matrix sign functions \cite{J.1980linear}-\cite{A.1991invariant}, and transport theory \cite{J.W.1999nonsymmetric}. For example, consider the linear time-invariant system represented by
\begin{align*}
\dot{\bp}(t)=A\bp(t)+B\bu(t),~~\bp(0)=\bp_0,
\end{align*}
where $B\in \bbC^{n\times m}$, $A$ is defined by \eqref{eq:CARE},
$\bp(t)\in\bbC^{n}$ and $\bu(t)\in\bbC^{m}$ are the state and control vectors of the system.
We can represent $u(t)=-R^{-1}B^{*}\bp(t)$ for a vector $\bx(t)$ such that the cost functional
\begin{align}
J(u)=\int^{t_1}_{t_0} [\bp^{*}(t)Q \bp(t)+\bu^*(t)R\bu(t)]\,dt+\bp^{*}(t_1)Z_1 \bp(t_1)
\label{eq:def-cost-J}
\end{align}
is minimized, where the cost penalty matrix $Q$ defined by \eqref{eq:CARE} is positive semi-definite, and
the cost penalty matrices $R,~Z_1$ are positive definite.
Taken $K=BR^{-1}B^*$, defined by \eqref{eq:CARE}, the state equation becomes
\begin{align*}
\dot{\bp}(t)=A\bp(t)-K\bx(t).
\end{align*}
The cost function reduces to
\begin{align*}
\hat{J}(u)=\int^{t_1}_{t_0} [\bp^{*}(t)Q \bp(t)+\bx^*(t)K\bx(t)]\,dt+\bp^{*}(t_1)Z_1 \bp(t_1).
\end{align*}
We wish to choose $\bx(t)$ to minimize $\hat{J}$.
This can be accomplished by recognizing the fact input
$\bx(t)+\delta\hat{\bx}(t)$,
which is some deviation from the desired stabilizing input.
The corresponding state is $\bp(t)+\delta\hat{\bp}(t)$ (by linearity).
It can be shown that $J$ is a minimum when the derivative of the adjusted
cost function $\hat{J}$ with respect to $\delta$ is zero \cite{J.2008Riccati}, \textit{i.e.,}
\begin{align*}
\int^{t_1}_{t_0} [\bp^{*}(t)Q \hat{\bp}(t)+\bx^*(t)A\hat{\bp}(t)-\bx^*(t)\dot{\hat{\bp}}(t)]\,dt
+\bp^{*}(t_1)Z_1 \hat{\bp}(t_1)=0.
\end{align*}
Using integration by parts and $\hat{\bp}(t_0)=0$, we have
\begin{align*}
\int^{t_1}_{t_0} [\bp^{*}(t)Q+\bx^*(t)A+\dot{\bx}^*(t)]\hat{\bp}(t)\,dt
+[\bp^{*}(t_1)Z_1-\bx^*(t_1)] \hat{\bp}(t_1)=0.
\end{align*}
If the system is controllable, then a suitable choice of $\hat{\bx}(t)$
gives any $\hat{\bp}(t)$. This leads to the
following requirements
\begin{align*}
\begin{cases}
\bp^{*}(t)Q+\bx^*(t)A+\dot{\bx}^*(t)=0,\\
\bp^{*}(t_1)Z_1-\bx^*(t_1)=0.
\end{cases}
\end{align*}
Now the system becomes
\begin{align*}
\dot{\bp}(t)=A\bp(t)-K\bx(t).
\end{align*}
Since there exists a state transition matrix by linearity
\begin{align*}
X(t,t_1)=
\begin{pmatrix}
X_{11}(t,t_1) & X_{12}(t,t_1)\\
X_{21}(t,t_1) & X_{22}(t,t_1)
\end{pmatrix}
\end{align*}
such that
\begin{align*}
\begin{pmatrix}
\bp(t)\\
\bx(t)
\end{pmatrix}
=
\begin{pmatrix}
X_{11}(t,t_1) & X_{12}(t,t_1)\\
X_{21}(t,t_1) & X_{22}(t,t_1)
\end{pmatrix}
\begin{pmatrix}
\bp(t_1)\\
\bx(t_1)
\end{pmatrix}.
\end{align*}
Then, we have $\bx(t)=X(t)\bp(t)$, where $X(t)$ is the solution
of the continuous time differential Riccati equation \cite{J.2008Riccati}
\begin{align}
-\frac{dX(t)}{dt}=A^*X(t)+X(t)A-X(t)KX(t)+Q.
\label{eq:dCARE}
\end{align}
When $t_1\rightarrow \infty$, we get a steady state stabilizing solution $\dot{X}(t)=0$.
The continuous time differential
Riccati equation \eqref{eq:dCARE} reduces to the CARE \eqref{eq:CARE}. The optimal choice of input that minimizes the infinite horizon cost function
\begin{align*}
J(u)=\int^{\infty}_{t_0} [\bp^{*}(t)Q \bp(t)+\bu^*(t)R\bu(t)]\,dt
\end{align*}
is $\bu(t)=-R^{-1}B^*X\bp(t)$.

Therefore, there are many scholars pay much attention to studying the  CARE \eqref{eq:CARE}. Many works are concerned with the existence of
positive (semi-) definite solution for this equation.
In 1961, Kalman derived the existence condition
of the positive definite solution using observability
and controllability \cite{R.1961new}. Subsequently, the existence conditions of the positive (semi-) definite solution have been
considerably investigated under suitable assumptions in \cite{W.1968matrix}-\cite{D.B.B.2012stabilizing}.

\begin{thm}
\cite{D.B.B.2012stabilizing}
For the CARE \eqref{eq:CARE}, if the pair $(A,K)$ is stabilizable
and $(A,Q)$ is detectable, then this equation has a unique Hermitian positive semi-definite
solution $X$. $(A,K)$ is called stabilizable if there exists a matrix
$F\in\bbC^{n\times n}$ such that $A-KF$ is stable, \textit{i.e.,}
all eigenvalues are in the open left half-plane $\bbC_{-}$. The pair $(A,Q)$ is called detectable if $(A^*,Q^*)$ is stabilizable.
\end{thm}

Additionally, there are many works focusing on the numerical algorithms for the CARE \eqref{eq:CARE}, such as
the Schur method \cite{R.B.D.2008parallel},
the matrix sign function \cite{R.1987solving},
the structure-preserving doubling algorithm \cite{E.H.W.W.2004structure},
and Krylov subspace projection method \cite{B.S.2013numerical}-\cite{B.H.J.2016preconditioned}. Kleinman \cite{D.1968iterative}, Banks and Ito \cite{H.K.1991numerical} applied the Newton method, due to its quadratic convergence, to solve this equation.
Nevertheless, at each Newton iteration step, a Lyapunov equation needs to be solved to get next iteration solution.
Therefore, Navasca and Morris in \cite{K.C.2005solution}-\cite{K.C.2006iterative}
combined the Newton method with a modified alternating-direction implicit (ADI) method.
Benner and his coauthors used the variant of Newton-ADI algorithm to
solve large-scale Riccati equations \cite{P.H.J.2008parameter}-\cite{P.J.2010Galerkin}.
Moreover, Feitzinger et al. in \cite{F.T.E.2009inexact}-\cite{P.M.2016inexact} proposed
and analyzed the inexact Newton-ADI method. Recently, \cite{M.A.2021Hermitian} has proposed an iteration
 scheme and given its theoretical analysis. It should be noticed that this algorithm still needs to solve a Lyapunov equation twice at each iteration step. Based on this, we concentrate on solving the CARE \eqref{eq:CARE} by combing the Newton method
with the generalized alternating directions implicit (GADI).

 The remaining part of this paper is organized as follows.
In Section \ref{sec:Newon-GADI}, we propose two Newton-GADI methods, including the Newton-GADI
and inexact Newton-GADI algorithms. Moreover, we show the convergence analysis of the Newton-GADI.  A practical method is given to select the quasi-optimal parameter.
In Section \ref{sec:NE}, some numerical examples are devoted to showing the effectiveness of
the proposed algorithms.  We draw some conclusion and remarks in the last section.

Throughout this paper, let $\bbC^{n\times m}$ be the set of all $n\times m$ complex matrices. $I_n$ is the identity matrix of order $n$.
If $A\in \bbC^{n\times n}$, the symbols $A^{\ast}$, $A^{-1}$ and $\|A\|_{2}$ express the conjugate transpose,
the inverse, the spectral norm of $A$, respectively. The eigenvalue and singular value sets of $A$ are denoted as $\Lambda(A)=\{\lambda_{i}(A),~i=1, 2,\cdots, n\},
~~\Sigma(A)=\{\sigma_{i}(A),~i=1, 2,\cdots, n\},$ where $\lambda_{i}(A)$ and $\sigma_i(A)$ represent
the $i$-th component, arranged in non-increasing order.  $\rho(A)=\max\limits_{1\leq i\leq n}\{|\lambda_{i}(A)|\}$
represents the spectral radius of $A$.

\section{Newton-GADI algorithms}
\label{sec:Newon-GADI}

It is well known that one classical approach to solving the  CARE \eqref{eq:CARE} is to tackle its nonlinearity
with a Newton-type method. In this section, we apply the Newton method to transform this equation into a Lyapunov equation.  Then we use the GADI method to solve the Lyapunov equation. The whole framework is denoted by the Newton-GADI.
To save computational amount, the inexact Newton-GADI method is presented. Further, we give the convergence analysis of the Newton-GADI algorithm.

\subsection{Iteration schemes}

Now we introduce the usual scheme of the Newton method
for solving the CARE \eqref{eq:CARE}.

Define the mapping $\calR: \bbC^{n\times n}\rightarrow \bbC^{n\times n}$:
\begin{align*}
\calR(X)=A^*X+XA-XKX+Q,~X\in \bbC^{n\times n},
\end{align*}
where $A,~K,~Q$ are defined in \eqref{eq:CARE}. The first Fr$\acute{\mbox{e}}$chet derivative of $\calR$
at a matrix $X$ is a linear map
$\calR'_X: \bbR^{n\times n}\rightarrow \bbR^{n\times n}$ given by
\begin{align*}
\calR'_X{(E)}=(A-KX)^*E+E(A-KX).
\end{align*}
The Newton method for the CARE \eqref{eq:CARE} is
\begin{align}
X_{k+1}=X_{k}-(\calR'_{X_k})^{-1}\calR(X_k),~k=0,1,\cdots,
\label{eq:Newton-scheme}
\end{align}
given that the map $\calR'_{X_k}$ is invertible.
The Newton iteration scheme \eqref{eq:Newton-scheme} is equivalent to
\begin{align}
(A-KX_k)^*X_{k+1}+X_{k+1}(A-KX_k)+X_k K X_k+Q=0.
\label{eq:Newton-CARE}
\end{align}
The convergence property for the Newton scheme \eqref{eq:Newton-CARE} is illustrated by the following result.

\begin{thm}
\cite{D.1968iterative}
For the CARE \eqref{eq:CARE}, assume that $(A,K)$ is stabilizable and $(A,Q)$ is detectable,
starting with any Hermitian matrix $X_0$ such that $A-KX_0$ is stable,
the sequence of Hermitian matrices $\{X_{k}\}_{k=1}^{\infty}$
determined by \eqref{eq:Newton-CARE}
quadratic converges to the Hermitian
positive semi-definite solution $X_*$ of \eqref{eq:CARE}.
In other words, there exists a constant $\delta>0$ such that
\begin{align*}
\|X_k-X_*\|\leq \delta \|X_{k-1}-X_*\|^2,~~k=0,1,\cdots,
\end{align*}
where $\|\cdot\|$ is any given matrix norm.
Moreover, the iteration sequence  $\{X_{k}\}_{k=1}^{\infty}$ has a monotone convergence behavior, \textit{i.e.,}
\begin{align*}
0\leq X_*\leq \cdots \leq X_{k+1}\leq X_{k}\leq \cdots \leq X_1.
\end{align*}
\end{thm}

Denote $A_k=KX_k-A$ and $F(X_k)=X_k K X_k+Q$, where the mapping $F: \bbC^{n\times n}\rightarrow \bbC^{n\times n}$.
Actually, due to $A-KX_k$ is stable, then $A_k$ is stable,
\textit{i.e.,} $\mbox{Re}(\lambda_i(A_k))>0$.
The iteration scheme can be written as
\begin{align}
A_k^*X_{k+1}+X_{k+1}A_k=F(X_k).
\label{eq:Newton-Lyapunov}
\end{align}
It is required to solve the Lyapunov equation \eqref{eq:Newton-Lyapunov}
in each Newton iteration step.
We recall an existence and uniqueness
theorem for this equation.
\begin{thm}
The Lyapunov equation
\begin{align*}
\hat{A}^*Y+Y\hat{A}+Z=0
\end{align*}
has a uniquely Hermitian positive semi-definite solution $Y\in\bbC^{n\times n}$ if and only if $\hat{A}\in\bbC^{n\times n}$ is stable, for any Hermitian matrix $Z\in\bbC^{n\times n}$.

\end{thm}
We apply the GADI framework in \cite{K.X.J.2021general} to solve the Lyapunov equation \eqref{eq:Newton-Lyapunov}. The GADI scheme is
\begin{align}
\begin{cases}
(\alpha_k I+ A_k^*) X_{k+1,\ell+\frac{1}{2}}=X_{k+1,\ell}(\alpha_{k} I-A_k)+F(X_k),\\
X_{k+1,\ell+1}(\alpha_{k} I+ A_k) =X_{k+1,\ell}[A_k-(1-\omega_k)\alpha_{k} I]
+(2-\omega_{k})\alpha_{k} X_{k+1,\ell+\frac{1}{2}},
\end{cases}
\label{eq:Newton-GADI}
\end{align}
where $\ell=0,1,\cdots$, $X_{k+1,0}=X_{k}$, $\alpha_{k} >0$ and $\omega_{k}\in [0,2)$.

Combining the Newton scheme with the GADI scheme, we have the Newton-GADI algorithm, summarized in Algorithm \ref{alg:Newton-GADI} below.

\begin{algorithm}[H]
\caption{(Newton-GADI Algorithm)}
\label{alg:Newton-GADI}
~~~Given the matrices $A,~K,~Q$, the integer numbers $k_{\max}>1$ and $\ell_{\max}>1$,
the outer and inter iteration tolerance $\varepsilon_{out}$
and $\varepsilon_{inn}$, respectively.

~~~Step 1. Compute $\beta=1+\|A\|_{\infty},~~B=A+\beta I$
and the initial matrix $X_0=X^{-1}$
by solving Lyapunov equation $B^*X+XB-2Q=0$.

~~~Step 2. Set $X_{k+1,0}=X_{k}$
and compute the matrices
\begin{align*}
A_k=KX_k-A,~~F(X_k)=X_k K X_k+Q.
\end{align*}

~~~Step 3. Solve the GADI iteration scheme \eqref{eq:Newton-GADI} to obtain $X_{k+1}=X_{k+1,\ell_k}$ such that
\begin{align*}
\|A^*X_{k+1,\ell_k}+X_{k+1,\ell_k}A-X_{k+1,\ell_k}KX_{k+1,\ell_k}+Q\|< \varepsilon_{inn}
~~\mbox{or}~~
\ell_k>\ell_{\max}.
\end{align*}

~~~Step 4. Compute the normalized residual
\begin{align*}
\mbox{NRes}(X_{k+1})
=\frac{\|A^*X_{k+1}+X_{k+1}A-X_{k+1}KX_{k+1}+Q\|_2}
{\|A^*X_{k+1}\|_2+\|X_{k+1}A\|_2+\|X_{k+1}KX_{k+1}\|_2+\|Q\|_2}.
\end{align*}

~~~Step 5. If $\mbox{NRes}(X_{k+1})< \epsilon_{out}$ or $k>k_{max}$,
then the approximate solution of the CARE \eqref{eq:CARE} is $\tilde{X}=X_{k+1}$;
else set $k=k+1$ and return to Step 2.
\end{algorithm}

\begin{remark}
From Algorithm \ref{alg:Newton-GADI}, it can be seen that when $\omega_{k}=0$, the GADI method naturally reduces to
the ADI method in \cite{E.1988iterative}. We adopt the method given in \cite{S.2011matrix} to select the  desired initial matrix $X_0$,
see Step 1 in Algorithm \ref{alg:Newton-GADI}.

\end{remark}

For large-scale Lyapunov equations at each Newton step,
it is important to control the accuracy of the solution to gain efficiency,
and wish to keep the overall fast convergence property of the Newton method.
The inexact Newton method proposed in \cite{F.T.E.2009inexact}
has shown a rigorous guideline for the termination of the inner iteration while the fast local rate of convergence is retained.
The Lyapunov equation can be approximately solved such that
\begin{align*}
&\|(A-KX_k)^*X_{k+1}+X_{k+1}(A-KX_k)+X_k K X_k+Q\|
\leq \eta_k\|\calR(X_k)\|\\
&=\eta_k\|A^*X_k+X_kA-X_kKX_k+Q\|.
\end{align*}
Formally, the iteration scheme is determined by solving
\begin{align}
(A-KX_k)^*X_{k+1}+X_{k+1}(A-KX_k)+X_k K X_k+Q=R_k.
\label{eq:Inexact-Newton-CARE}
\end{align}
The following convergence property for the inexact Newton method is referenced in \cite{F.T.E.2009inexact}.

\begin{thm}
\cite{F.T.E.2009inexact}
For the CARE \eqref{eq:CARE}, assume that $(A,K)$ is stabilizable and $(A,Q)$ is detectable,
and $X_*\in\bbC^{n\times n}$ is the Hermitian positive semi-definite solution. If there exist $\tilde{\delta}>0$ and $\tilde{\eta}>0$
such that for an initial matrix $X_0\in\bbC^{n\times n}$ meeting
$\|X_0-X_*\|\leq \tilde{\delta}$, then the iteration sequence $X_k$ generated by \eqref{eq:Inexact-Newton-CARE}
converges to $X_*$ if the residual $R_k$ satisfies
\begin{align*}
\|R_k\|\leq \eta_k \|A^*X_k+X_kA-X_kKX_k+Q\|.
\end{align*}
The convergence rate is linear if $\eta_k\in (0,\tilde{\eta}]$.
It is super-linear if $\eta_k\rightarrow 0$ and is quadratic if
$$ \eta_k\leq K_{\eta} \|A^*X_k+X_kA-X_kKX_k+Q\|$$
for some $K_{\eta}>0$.
\end{thm}

As a consequence,
the inexact Newton-GADI method
is summarized in Algorithm \ref{alg:Inexact-Newton-GADI}. Obviously, this algorithm consists of approximating
the solution in the $k$-th Newton step for
a certain tolerance $\eta_k>0$. It decreases when $k$ increases,
\textit{e.g.}, set $\eta_k=1/(k^3+1)$ (see Section \ref{sec:NE} for details).

\begin{algorithm}[H]
\caption{ (Inexact Newton-GADI Algorithm)}
\label{alg:Inexact-Newton-GADI}
~~~Given the matrices $A,~K,~Q$, the integer numbers $k_{\max}>1$ and $\ell_{\max}>1$,
a sequence of positive numbers $\{\eta_k\}$ such that $\eta_{k+1}<\eta_k$ for all $k$,
the outer iteration tolerance $\varepsilon_{out}$.

~~~Step 1. Compute $\beta=1+\|A\|_{\infty},~~B=-(A+\beta I)$
and the initial matrix $X_0=X^{-1}$
by solving the Lyapunov equation $B^*X+XB-2Q=0$.

~~~Step 2. Set $X_{k+1,0}=X_{k}$
and compute the matrices
\begin{align*}
A_k=KX_k-A,~~F(X_k)=X_k K X_k+Q.
\end{align*}

~~~Step 3. Solve the GADI iteration scheme \eqref{eq:Newton-GADI} to obtain $X_{k+1}=X_{k+1,\ell_k}$
such that
\begin{align*}
\|R_k\|\leq \eta_k \|A^*X_k+X_kA-X_kKX_k+Q\|
~~\mbox{or}~~
\ell_k>\ell_{\max}.
\end{align*}


~~~Step 4. Compute the normalized residual
\begin{align*}
\mbox{NRes}(X_{k+1})
=\frac{\|A^*X_{k+1}+X_{k+1}A-X_{k+1}KX_{k+1}+Q\|_2}
{\|A^*X_{k+1}\|_2+\|X_{k+1}A\|_2+\|X_{k+1}KX_{k+1}\|_2+\|Q\|_2}.
\end{align*}

~~~Step 5. If $\mbox{NRes}(X_{k+1})< \epsilon_{out}$ or $k>k_{max}$,
then the approximate solution of the CARE \eqref{eq:CARE} is $\tilde{X}=X_{k+1}$;
else set $k=k+1$ and return to Step 2.
\end{algorithm}

\subsection{Convergence analysis}

In this subsection, we prove that the Newton-GADI scheme is convergent. The convergence analysis of the Inexact Newton-GADI scheme is similar and hence is omitted. We recall the definition and property of Kronecker product.

For $A_{1}=(a_{ij})\in\bbC^{n_A \times m_A}$
and $B_{1}\in\bbC^{n_B \times m_B}$,
the Kronecker product is defined as
\begin{align*}
A_{1}\otimes B_{1}=\begin{pmatrix}
a_{11}B_{1} & a_{12}B_{1} & \cdots & a_{1m_A}B_{1}\\
a_{21}B_{1} & a_{22}B_{1} & \cdots & a_{2m_A}B_{1}\\
\vdots &  \vdots  &         & \vdots\\
a_{n_A 1}B_{1} & a_{n_A2}B_{1} & \cdots & a_{n_Am_A}B_{1}
\end{pmatrix}
\in \bbC^{(n_An_B)\times (m_Am_B)}.
\end{align*}
The vectorization operator (abbreviated as vec): $\bbC^{m\times n}\rightarrow \bbC^{mn}$ is
\begin{align*}
\mbox{vec}(X_{1})=(x^{T}_1,x^{T}_2,\cdots,x^{T}_n)^{T}\in\bbC^{mn},
~~\mbox{with}~~X_{1}=(x_1,x_2,\cdots,x_n)\in\bbC^{m\times n}.
\end{align*}
A brief review of some properties related to Kronecker products is required.

\begin{proposition}
For $A_{1},~B_{1},~C,~D,~X_{1}$ with appropriate dimensions, we have

(i) $(A_{1}\otimes B_{1})(C\otimes D)=(A_{1}C)\otimes(B_{1}D)$;

(ii) $(A_{1}\otimes B_{1})^{-1}=A_{1}^{-1}\otimes B_{1}^{-1}$;

(iii) $\Lambda(A_{1}\otimes B_{1})=\{\lambda\mu:~\lambda\in\Lambda(A_{1}) ~\mbox{and}~
\mu\in \Lambda(B_{1})\}$;

(iv) $(A_{1}\otimes B_{1})x=\mbox{vec}(B_{1}^TX_{1}A_{1})$,~$\mbox{vec}(X_{1})=x$.
\end{proposition}
Thereby, we have the following convergence conclusion for the Newton-GADI method.

\begin{thm}
\label{thm:GADIconver}
For the CARE \eqref{eq:CARE}, assume that $(A,K)$ is stabilizable and $(A,Q)$ is detectable,
the parameters $\alpha_{k}>0$ and $0\leq \omega_{k}<2$.
Then, for any $k=0,1,\cdots$,
the iteration sequence $\{X_{k+1,\ell}\}^{\infty}_{\ell=0}$
defined by \eqref{eq:Newton-GADI} converges to $X_{k+1}$.
\end{thm}

\begin{proof}
Taking the operation ``vec" from the GADI iteration scheme \eqref{eq:Newton-GADI}
yields
\begin{align}
\begin{cases}
(\alpha_{k} I+ I\otimes A_k^*) \mbox{vec}(X_{k+1,\ell+\frac{1}{2}})
=(\alpha_{k} I-A^T_k\otimes I)\mbox{vec}(X_{k+1,\ell})+\mbox{vec}(F(X_k)),\\
(\alpha_{k} I+ A^T_k\otimes I)\mbox{vec}(X_{k+1,\ell+1})
=[A^T_k\otimes I-(1-\omega_{k})\alpha_{k} I]\mbox{vec}(X_{k+1,\ell})
+(2-\omega_{k})\alpha_{k} \cdot \mbox{vec}(X_{k+1,\ell+\frac{1}{2}}).
\end{cases}
\label{eq:CARE-GADI-vec-1}
\end{align}
Denote $x=\mbox{vec}(X)$ and
$f(X_k)=\mbox{vec}(F(X_k))$.
The iteration scheme \eqref{eq:CARE-GADI-vec-1} is equivalent to
\begin{align}
\begin{cases}
(\alpha_{k} I+ I\otimes A_k^*) x_{k+1,\ell+\frac{1}{2}}
=(\alpha_{k} I-A^T_k\otimes I) x_{k+1,\ell}+ f(X_k),\\
(\alpha_{k} I+ A^T_k\otimes I) x_{k+1,\ell+1}
=[A^T_k\otimes I-(1-\omega_{k})\alpha_{k} I] x_{k+1,\ell}
+(2-\omega_{k})\alpha_{k} x_{k+1,\ell+\frac{1}{2}}.
\end{cases}
\label{eq:CARE-GADI-vec}
\end{align}
As $\alpha_{k}>0$ and $A_k$ is stable,
then $\alpha_{k} I+ I\otimes A_k^*$ and $\alpha_{k} I+ A^T_k\otimes I$
are nonsingular. By \eqref{eq:CARE-GADI-vec}, we get
\begin{align*}
x_{k+1,\ell+1}
&=(\alpha_{k} I+ A^T_k\otimes I)^{-1}
\Big\{ \Big[A^T_k\otimes I-(1-\omega_{k})\alpha_{k} I \Big]
+(2-\omega_{k})\alpha_{k}
(\alpha_{k} I+ I\otimes A_k^*)^{-1}(\alpha_{k} I-A^T_k\otimes I)\Big\} \\
&~~~~~~~~~~~~~~~~~~~~~~~~~~
x_{k+1,\ell}+ (2-\omega_{k})\alpha_{k} (\alpha_{k} I+ A^T_k\otimes I)^{-1}
(\alpha_{k} I+ I\otimes A_k^*)^{-1}f(X_k)\\
&=(\alpha_{k} I+ A^T_k\otimes I)^{-1}(\alpha_{k} I+ I\otimes A_k^*)^{-1}
\Big[\alpha_{k}^2 I+ (I\otimes A_k^*)( A^T_k\otimes I)
-(1-\omega_{k})\alpha_{k} C_k
\Big]
x_{k+1,\ell}\\
&~~~~~~~~~~~~~~~~~~~~~~~~~~
+ (2-\omega_{k})\alpha_{k} (\alpha_{k} I+ A^T_k\otimes I)^{-1}
(\alpha_{k} I+ I\otimes A_k^*)^{-1}f(X_k)
\end{align*}
where $C_k=I\otimes A_k^*+A^T_k\otimes I$.
Denote
\begin{align*}
&T_k(\alpha_{k},\omega_{k})=(\alpha_{k} I+ A^T_k\otimes I)^{-1}(\alpha_{k} I+ I\otimes A_k^*)^{-1}
\Big[\alpha_{k}^2 I+ (I\otimes A_k^*)( A^T_k\otimes I)
-(1-\omega_{k})\alpha_{k} C_k \Big],\\
&G_k(\alpha_{k},\omega_{k})=(2-\omega_{k})\alpha_{k} (\alpha_{k} I+ A^T_k\otimes I)^{-1}
(\alpha_{k} I+ I\otimes A_k^*)^{-1},
\end{align*}
thus
\begin{align}
x_{k+1,\ell+1}=T_k(\alpha_{k},\omega_{k})x_{k+1,\ell}+G_k(\alpha_{k},\omega_{k})f(X_k).
\label{eq:iteration scheme}
\end{align}
Next, we prove that the iteration matrix $\rho(T_k(\alpha_{k},\omega_{k}))<1$ for $0<\alpha_{k}$
and $0\leq \omega_{k}<2$. Note that
\begin{align*}
2\alpha_{k} C_k=
-(\alpha_{k} I- I\otimes A_k^*)(\alpha_{k} I- A^T_k\otimes I)
+(\alpha_{k} I+ I\otimes A_k^*)(\alpha_{k} I+ A^T_k\otimes I),
\end{align*}
then
\begin{align}
T_k(\alpha_{k},\omega_{k})&=(\alpha_{k} I+ A^T_k\otimes I)^{-1}(\alpha_{k} I+ I\otimes A_k^*)^{-1}
\Big[\alpha_{k}^2 I+ (I\otimes A_k^*)( A^T_k\otimes I)
-(1-\omega_{k})\alpha_{k} C_k \Big]\notag\\
&=(\alpha_{k} I+ A^T_k\otimes I)^{-1}(\alpha_{k} I+ I\otimes A_k^*)^{-1}
\Big[(\alpha_{k} I- I\otimes A_k^*)(\alpha_{k} I- A^T_k\otimes I)
+\omega_{k}\alpha_{k} C_k \Big]\notag\\
&=\frac{1}{2}\Big\{(2-\omega_{k})
(\alpha_{k} I+ A^T_k\otimes I)^{-1}(\alpha_{k} I+ I\otimes A_k^*)^{-1}
(\alpha_{k} I- I\otimes A_k^*)(\alpha_{k} I- A^T_k\otimes I)
+\omega_{k} I
\Big\}\notag\\
&=\frac{1}{2}[(2-\omega_{k})T_k(\alpha_{k})+\omega_{k} I]
\label{eq:T-alpha-omega}
\end{align}
with
\begin{align}
T_k(\alpha_{k})=(\alpha_{k} I+ A^T_k\otimes I)^{-1}(\alpha_{k} I+ I\otimes A_k^*)^{-1}
(\alpha_{k} I- I\otimes A_k^*)(\alpha_{k} I- A^T_k\otimes I).
\label{eq:T-alpha}
\end{align}
Using \eqref{eq:T-alpha-omega}, it is evident that
\begin{align*}
\lambda_i(T_k(\alpha_{k},\omega_{k}))
=\frac{1}{2}[(2-\omega_{k})\lambda_i(T_k(\alpha_{k}))+\omega_{k} ](i=1,2,\cdots,n),
\end{align*}
and
\begin{align*}
\rho(T_k(\alpha_{k},\omega_{k}))
\leq \frac{1}{2}[(2-\omega_{k})\rho(T_k(\alpha_{k}))+\omega_{k} ].
\end{align*}
Obviously, the matrix $T_k(\alpha_{k})$ defined in \eqref{eq:T-alpha} is similar to
\begin{align*}
\hat{T}_k(\alpha_{k})=(\alpha_{k} I+ I\otimes A_k^*)^{-1}
(\alpha_{k} I- I\otimes A_k^*)(\alpha_{k} I- A^T_k\otimes I)
(\alpha_{k} I+ A^T_k\otimes I)^{-1}
\end{align*}
through the matrix $\alpha_{k} I+ A^T_k\otimes I$.
Thus we can obtain
\begin{align}
\rho(T_k(\alpha_{k}))
&\leq
\|(\alpha_{k} I+ I\otimes A_k^*)^{-1}(\alpha_{k} I- I\otimes A_k^*)\|_2\cdot
\|(\alpha_{k} I- A^T_k\otimes I)(\alpha_{k} I+ A^T_k\otimes I)^{-1}\|_2 \notag\\
&\leq
\|(\alpha_{k} I+ I\otimes A_k^*)^{-1}(\alpha_{k} I- I\otimes A_k^*)\|_2\cdot
\|(\alpha_{k} I- A^T_k\otimes I)(\alpha_{k} I+ A^T_k\otimes I)^{-1}\|_2 \notag\\
&=\|A^L_k\|_2\|A^R_k\|_2,
\label{eq:boundrhoALR}
\end{align}
where $A^L_k=(\alpha_{k} I+ I\otimes A_k^*)^{-1}(\alpha_{k} I- I\otimes A_k^*)$
and $A^R_k=(\alpha_{k} I- A^T_k\otimes I)(\alpha_{k} I+ A^T_k\otimes I)^{-1}$.
Then, for $\bx\in\bbC^{n^2\times 1}$, we get
\begin{align*}
\|A^L_k\|^2_2
&=\max_{\|\bx\|_2=1}
\frac{\|(\alpha_{k} I- I\otimes A_k^*)\bx\|^2_2}
{\|(\alpha_{k} I+I\otimes A_k^*)\bx\|^2_2}\\
&=\max_{\|\bx\|_2=1}
\frac{\|(I\otimes A_k^*)\bx\|^2_2-\alpha_{k} \bx^*(I\otimes A_k+I\otimes A_k^*)\bx+\alpha^2_k}
{\|(I\otimes A_k^*)\bx\|^2_2+\alpha_{k} \bx^*(I\otimes A_k+I\otimes A_k^*)\bx+\alpha^2_k}\\
&\leq \max_{\|\bx\|_2=1}
\frac{\|(I\otimes A_k^*)\bx\|^2_2- 2\alpha_{k} \min_{i} Re(\lambda_i(I\otimes A_k))+\alpha^2_k}
{\|(I\otimes A_k^*)\bx\|^2_2+2\alpha_{k} \min_{i} Re(\lambda_i(I\otimes A_k))+\alpha^2_k}\\
&\leq \frac{\|I\otimes A_k^*\|^2_2- 2\alpha_{k} \min_{i} Re(\lambda_i(I\otimes A_k))+\alpha^2_k}
{\|I\otimes A_k^*\|^2_2+2\alpha_{k} \min_{i} Re(\lambda_i(I\otimes A_k))+\alpha^2_k}\\
&= \frac{\|I\otimes A_k^*\|^2_2- 2\alpha_{k} \min_{i} Re(\lambda_i(A_k))+\alpha^2_k}
{\|I\otimes A_k^*\|^2_2+2\alpha_{k} \min_{i} Re(\lambda_i(A_k))+\alpha^2_k}.
\end{align*}
Analogously, we obtain
\begin{align*}
\|A^R_k\|^2_2
\leq \frac{\|I\otimes A_k^T\|^2_2- 2\alpha_{k} \min_{i} Re(\lambda_i(A_k))+\alpha^2_k}
{\|I\otimes A_k^T\|^2_2+2\alpha_{k} \min_{i} Re(\lambda_i(A_k))+\alpha^2_k}.
\end{align*}
Since $A_k$ is stable and $\alpha_{k}>0$, by \eqref{eq:boundrhoALR}, it follows that $\|A^L_k\|_2<1$ and $\|A^R_k\|_2<1$.
Therefore,
\begin{align*}
\rho(T_k(\alpha_{k}))< 1,~~ \mbox{and}~~
\rho(T_k(\alpha_{k},\omega_{k}))
\leq \frac{1}{2}[(2-\omega_{k})\rho(T_k(\alpha_{k}))+\omega_{k} ]<1,
\end{align*}
which implies that the iteration scheme \eqref{eq:iteration scheme} is convergent. The proof is completed.

\end{proof}

\subsection{Parameter selection}
\label{sec:parameter}


In this subsection, we first compare the convergence rate between the Newton-GADI and Newton-ADI methods.
That is, the relation between $\rho(T_k(\alpha_{k},\omega_{k}))$ and $\rho(T_k(\alpha_{k}))$ defined in Section 2.2 is discussed. Then we present a practical method  to select the quasi-optimal parameter.

\begin{thm}
\label{thm:relationADIandGADI}
Assume that the eigenvalues of $A_k$ have positive real parts.
For $j_1,~j_2\in\{1,\cdots,n^2\}$, denote
$$\rho(T_k(\alpha_{k},\omega_{k}))=\vert\zeta_{j_1}\vert=\vert a_{j_1}+i\, b_{j_1} \vert,~~\rho(T_k(\alpha_{k}))=\vert \eta_{j_2}\vert=\vert c_{j_2}+i\, d_{j_2} \vert.$$

(i) When $\vert \eta_{j_2}\vert^2\leq  c_{j_2}$, we have

\begin{align*}
\rho(T_k(\alpha_{k}))<\rho(T_k(\alpha_{k},\omega_{k}))<1.
\end{align*}

(ii) When $\vert \eta_{j_1}\vert^2> c_{j_1}$ and $0< \omega_{k} <\frac{4(\vert \eta_{j_1} \vert^2-c_{j_1})}{(1-c_{j_1})^2+d_{j_1}^2}<2,$
we have
\begin{align*}
\rho(T_k(\alpha_{k},\omega_{k}))<\rho(T_k(\alpha_{k}))<1.
\end{align*}
\end{thm}

\begin{proof}
For $\zeta_j=a_j+ib_j\in\Lambda(T_k(\alpha_{k},\omega_{k}))$, $\eta_j=c_j+id_j\in\Lambda(T_k(\alpha_{k}))(j=1,2,\cdots,n)$, then
\begin{align*}
2\vert\zeta_j \vert=\vert (2-\omega_{k})\eta_j +\omega_{k}\vert.
\end{align*}
Therefore,
\begin{align*}
4\vert\zeta_j \vert^2=\vert (2-\omega_{k})\eta_j +\omega_{k}\vert^2
=[(2-\omega_{k})c_j+\omega_{k}]^2+[(2-\omega_{k})d_j]^2
=(2-\omega_{k})^2\vert \eta_j\vert^2+\omega_{k}^2+2(2-\omega_{k})\omega_{k} c_j,
\end{align*}
\textit{i.e.,}
\begin{align*}
4\vert\zeta_j \vert^2-4\vert\eta_j \vert^2
&=(\omega_{k}^2-4\omega_{k})\vert \eta_j \vert^2+\omega_{k}^2+2(2-\omega_{k})\omega_{k} c_j\\
&=(1-2c_j+\vert\eta \vert^2)\omega_{k}^2-4(\vert \eta_j \vert^2-c_j)\omega_{k}\\
&=[(1-c_j)^2+d_j^2]\omega_{k}^2-4(\vert \eta_j \vert^2-c_j)\omega_{k}.
\end{align*}
Note that $$\rho(T_k(\alpha_{k},\omega_{k}))=\vert\zeta_{j_1}\vert,~~\rho(T_k(\alpha_{k}))=\vert \eta_{j_2}\vert.$$

(i) When $\vert \eta_{j_2}\vert^2\leq  c_{j_2}$, we have
\begin{align*}
\rho(T_k(\alpha_{k}))<\rho(T_k(\alpha_{k},\omega_{k}))<1.
\end{align*}

(ii) When $\vert \eta_{j_1}\vert^2> c_{j_1}$ and $0< \omega_{k} <\frac{4(\vert \eta_{j_1} \vert^2-c_{j_1})}{(1-c_{j_1})^2+d_{j_1}^2}<2,$ we have
\begin{align*}
\vert\zeta_{j_1} \vert<\vert \eta_{j_1}\vert,
\end{align*}
which means that $\rho(T_k(\alpha_{k},\omega_{k}))<\rho(T_k(\alpha_{k}))<1$.
\end{proof}

\begin{remark}
By Theorem \ref{thm:relationADIandGADI}, if $\vert \eta_{j_1}\vert^2> c_{j_1}$ and $0< \omega_{k} <\frac{4(\vert \eta_{j_1} \vert^2-c_{j_1})}{(1-c_{j_1})^2+d_{j_1}^2}<2,$ the Newton-GADI has a faster convergence rate than Newton-ADI proposed in \cite{E.1988iterative}. This is verified by numerical examples in Section \ref{sec:NE}.
\end{remark}

Next, we provide a method for selecting theoretical quasi-parameters by minimizing $\rho(T_k(\alpha_{k},\omega_{k}))$.
\begin{thm}
\label{thm:para}
Assume that the eigenvalues of $A_k$ have positive real parts, set
\begin{align}
\nu_n=\max_{j}\sigma_j(A_k), ~~
\min_j Re(\lambda_j(A_k))=Re(\mu_1)=Re(a_\mu+i b_\mu)=a_\mu.
\label{notation:numu}
\end{align}
Then the quasi-optimal parameter $\alpha_k$ can be obtained by
\begin{align}
\alpha_{k}^*=\mbox{arg}\min_{\alpha_{k}}
\frac{\nu_n^2- 2\alpha_{k} a_\mu+\alpha^2_k}
{\nu_n^2+2\alpha_{k} a_\mu+\alpha^2_k}
=\nu_n.
\label{para:alpha}
\end{align}
\end{thm}

\begin{proof}
Using the bound defined by \eqref{eq:boundrhoALR}, we have
\begin{align*}
\rho^2(T_k(\alpha_{k}))
&\leq \|A_L\|^2_2\|A_R\|^2_2\\
&\leq \frac{\|I\otimes A_k^*\|^2_2- 2\alpha_{k} \min_{i} Re(\lambda_i(A_k))+\alpha^2_k}
{\|I\otimes A_k^*\|^2_2+2\alpha_{k} \min_{i} Re(\lambda_i(A_k))+\alpha^2_k} \cdot
\frac{\|I\otimes A_k^T\|^2_2- 2\alpha_{k} \min_{i} Re(\lambda_i(A_k))+\alpha^2_k}
{\|I\otimes A_k^T\|^2_2+2\alpha_{k} \min_{i} Re(\lambda_i(A_k))+\alpha^2_k}\\
&=\frac{\nu_n^2- 2\alpha_{k} a_\mu+\alpha^2_k}
{\nu_n^2+2\alpha_{k} a_\mu+\alpha^2_k} \cdot
\frac{\nu_n^2- 2\alpha_{k}a_\mu+\alpha^2_k}
{\nu_n^2+2\alpha_{k}a_\mu +\alpha^2_k},
\end{align*}
\textit{i.e.,}
\begin{align*}
\rho(T_k(\alpha_{k}))\leq
\frac{\nu_n^2- 2\alpha_{k} a_\mu+\alpha^2_k}
{\nu_n^2+2\alpha_{k} a_\mu+\alpha^2_k},
\end{align*}
where $\nu_n$ and $a_\mu$ are given in \eqref{notation:numu}.
Denote
\begin{align*}
\Psi(\alpha_{k})=\frac{\nu_n^2- 2\alpha_{k} a_\mu+\alpha_{k}^2}
{\nu_n^2+2\alpha_{k} a_\mu+\alpha_{k}^2}.
\end{align*}
Then
\begin{align*}
\frac{\partial\Psi(\alpha_{k})}{\partial\alpha_{k}}
=\frac{4 a_\mu (\alpha_{k}^2-\nu_n^2)}
{(\nu_n^2+2\alpha_{k} a_\mu+\alpha_{k}^2)^2}.
\end{align*}
It is easy to show that $\Psi(\alpha_{k})$ gets to its minimum
when $\alpha_{k}=\nu_n$, \textit{i.e.,}
$$\alpha_{k}^*=\mbox{arg}\min\limits_{\alpha_k}\Psi(\alpha_k)=\nu_n,$$
where $\nu_n$ depends on the matrix $A_k$.
\end{proof}

\begin{remark}
In terms of Theorem \ref{thm:para}, we usually select $\omega_k=1$ and $\alpha_k^*$ given in \eqref{para:alpha}
as the quasi-optimal parameters. It can be seen that this parameter choice is reasonable using the residual analysis in Section \ref{sec:NE}.
\end{remark}

\section{Numerical experiments}
\label{sec:NE}

In this section, we use several examples of the quadratic optimal control
to show the numerical feasibility and effectiveness of the Newton-GADI and inexact Newton-GADI (InNewton-GADI) algorithms.
The whole process is performed on a computer with Intel Core 3.20GHz CPU, 4.00GB RAM and MATLAB R2017a. IT(inn) and IT(out) represent the number of inter iteration steps and
outer iteration steps, respectively. Denote IT(cumul) by the sum of inter and outer iteration steps,
IT(ave) by the number of the GADI iterative steps required by the average Newton step, and CPU by the computing time.
The notation $\mbox{Tri}(c,a,b)_{n\times n}$ demonstrates
the $n\times n$ size tridiagonal matrix
\begin{align*}
\begin{pmatrix}
a & b &        &  &\\
c & a & b      &   &\\
  & c & \ddots & \ddots&\\
  &   &\ddots  & \ddots &b\\
  &   &        &   c    &a
\end{pmatrix}.
\end{align*}

In our numerical experiment, set the outer iteration tolerance
$\varepsilon_{out}=1\times 10^{-8},~\varepsilon_{inn}=1\times 10^{-8}$,
$\ell_{\max}=1000,~k_{\max}=1000$.
In the inexact Newton-GADI algorithm, set $\eta_k=1/(k^3+1)$.
Denote the numerical Hermitian semi-definite solution as $\tilde{X}$, the finally normalized residual as
\begin{align*}
\mbox{NRes}(\tilde{X})
=\frac{\|A^*\tilde{X}+\tilde{X}A-\tilde{X}K\tilde{X}+Q\|_2}
{\|A^*\tilde{X}\|_2+\|\tilde{X}A\|_2+\|\tilde{X}K\tilde{X}\|_2+\|Q\|_2},
\end{align*}
and the error of the inter iteration in $k$-th Newton iteration step as
\begin{align*}
R(X_k)=A_k^*\tilde{X}^{in}_k+\tilde{X}^{in}_kA_k-X_kKX_k-Q,
\end{align*}
with $\tilde{X}^{in}_k$ is the numerical solution of Lyapunov equation
in $k$-th Newton iteration step.

When $\omega_{k}=0$, the Newton-GADI reduces to Newton-ADI proposed in \cite{E.1988iterative}.
It can be seen that the Newton-GADI converges faster than Newton-ADI in the following examples.

\begin{example}
Consider the linear time-invariant system of the form
\begin{align*}
\dot{\bp}(t)=A\bp(t)+B\bu(t)=\begin{pmatrix}
-2+10\,i & 0  & -1\\
0  & -1+10\,i & 0\\
-1 & -1 & -2\,i
\end{pmatrix}\bp(t)
+\begin{pmatrix}
-2 & 0 & -1\\
0 & -1 & -1\\
1 & 0 & -2
\end{pmatrix}\bu(t).
\end{align*}
We wish to minimize the cost function $J$ defined in \eqref{eq:def-cost-J}.
The cost penalty matrices are
\begin{align*}
R = \begin{pmatrix}
1 & 0  & 0\\
0  & 1 & 0\\
0 & 0 & 4
\end{pmatrix},~
Q=\begin{pmatrix}
0 & 0 & 0\\
0 & 1 & 0\\
0 & 0 & 5
\end{pmatrix}.
\end{align*}
Then
$$
K=BR^{-1}B^{T}=\begin{pmatrix}
17/4 & 1/4 & -3/2\\
1/4  & 5/4 & 1/2\\
-3/2 & 1/2 & 2
\end{pmatrix}.
$$
\end{example}
We can verify that $K$ is positive definite and $Q$ is semi-positive definite.
The pair $(A,K)$ is stabilize as
\begin{align*}
\Lambda(A-KI_3)=\{-1.9936-1.9179\,i, -6.2704+9.9858\,i, -2.2360 + 9.9321\,i\}.
\end{align*}
The pair $(A,Q)$ is detectable as
\begin{align*}
\Lambda(Q^*-A^*F)=\{  -2.0000 + 0.0000\,i, -3.5000 - 1.6583\,i, -3.5000 + 1.6583\,i\}
\end{align*}
with
\begin{align*}
F=\begin{pmatrix}
-0.2785 - 0.5729\,i &  0 &  0.2387 - 0.3660\,i\\
-0.1651 - 0.0121\,i & -0.0297 - 0.2970\,i &  0.2773 - 0.3460\,i\\
0.2865 - 1.6393\,i &  0 &  0.1830 + 3.1194\,i
\end{pmatrix}.
\end{align*}
The numerical Hermitian semi-definite solution is
\begin{align*}
\tilde{X}= \begin{pmatrix}
-0.0020+ 0.0120i & -0.0034 + 0.0121i  &-0.1269-0.0611i\\
-0.0034 + 0.0121i&   0.0001 + 0.0618i &-0.1251-0.0780i\\
-0.1269 - 0.0611i&  -0.1251 - 0.0780i & 1.1794-0.9689i
\end{pmatrix}.
\end{align*}

We know that the approximate solution of the Lyapunov equation
can be obtained by using the inexact algorithm at each Newton step.
When the error generated by the inexact algorithm is very large, it will affect
the convergence rate in next step.
Figure \ref{fig1:errorinn} displays the error $\|R(X_k)\|_2$ of
the inexact inter iterative algorithm in each Newton step.
Moreover, due to the large error $\|R(X_2)\|_2$ caused by the GADI in the second iteration step,
from Figure \ref{fig1:IT}, we know that the inner iteration step exceeds the maximum and is not convergent in the third Newton step.
Therefore, we reset $\eta_k=\frac{1}{k^4+1}$ in the inexact Newton-ADI algorithm,
see Figure \ref{fig1:errorinn2}.

\begin{figure}[!htbp]
\centering
\subfigure{
\includegraphics[width=6in]{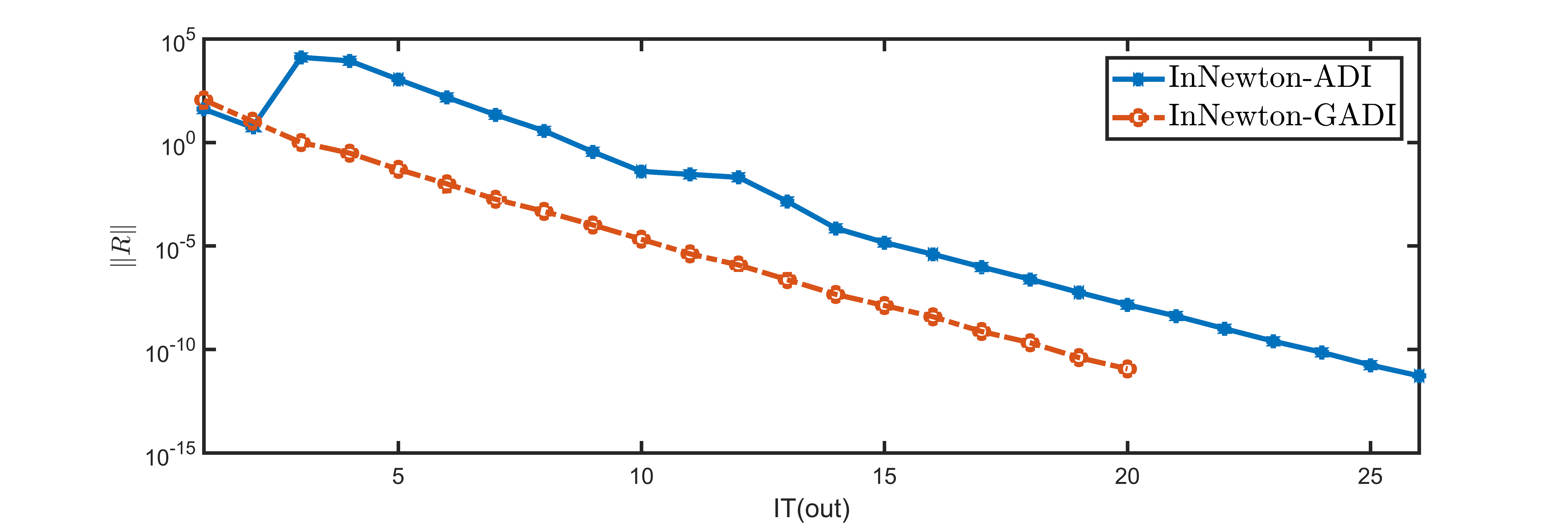}}
\caption{Comparing the inner iteration accuracy of two inexact algorithm with $\eta_k=\frac{1}{k^3+1}$.}\label{fig1:errorinn}
\end{figure}

\begin{figure}[!htbp]
\centering
\subfigure{
\includegraphics[width=6in]{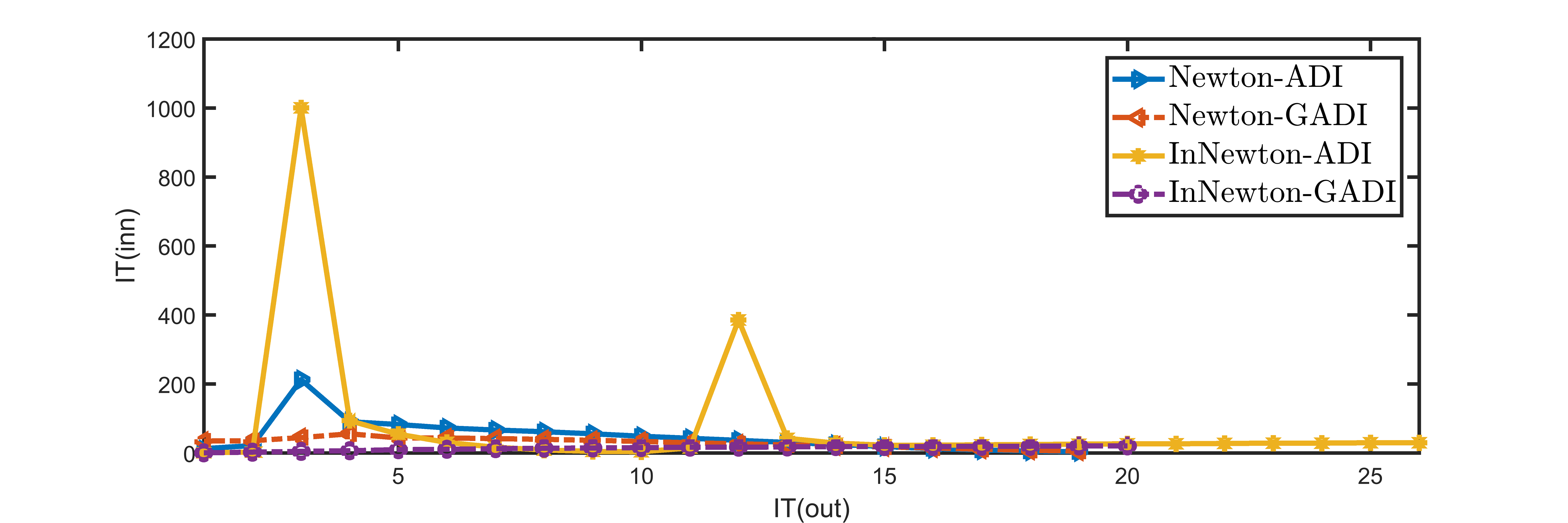}}
\caption{Comparing the number of inner iteration steps of
Newton-ADI, Newton-GADI, InNewton-ADI and InNewton-GADI at each Newton step
with $\eta_k=\frac{1}{k^3+1}$.}\label{fig1:IT}
\end{figure}

\begin{figure}[!htbp]
\centering
\subfigure{
\includegraphics[width=6in]{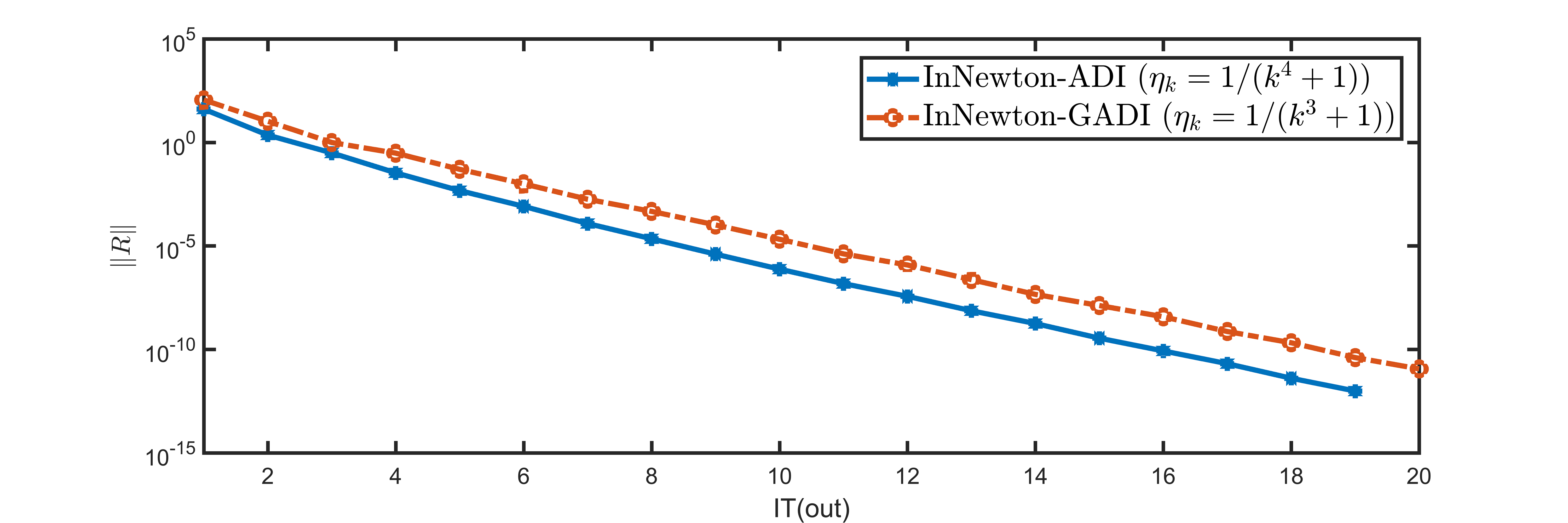}}
\caption{Comparing two inexact algorithms.}\label{fig1:errorinn2}
\end{figure}

Figure \ref{fig1:error2} shows the convergence process of the four algorithms,
including the Newton-ADI ($\omega=0$), the Newton-GADI,  the InNewton-ADI ($\omega=0$)
and the InNewton-GADI methods.
It is not difficult to find that these algorithms
have almost the same convergence rate.
\begin{figure}[!htbp]
\centering
\subfigure{
\includegraphics[width=6in]{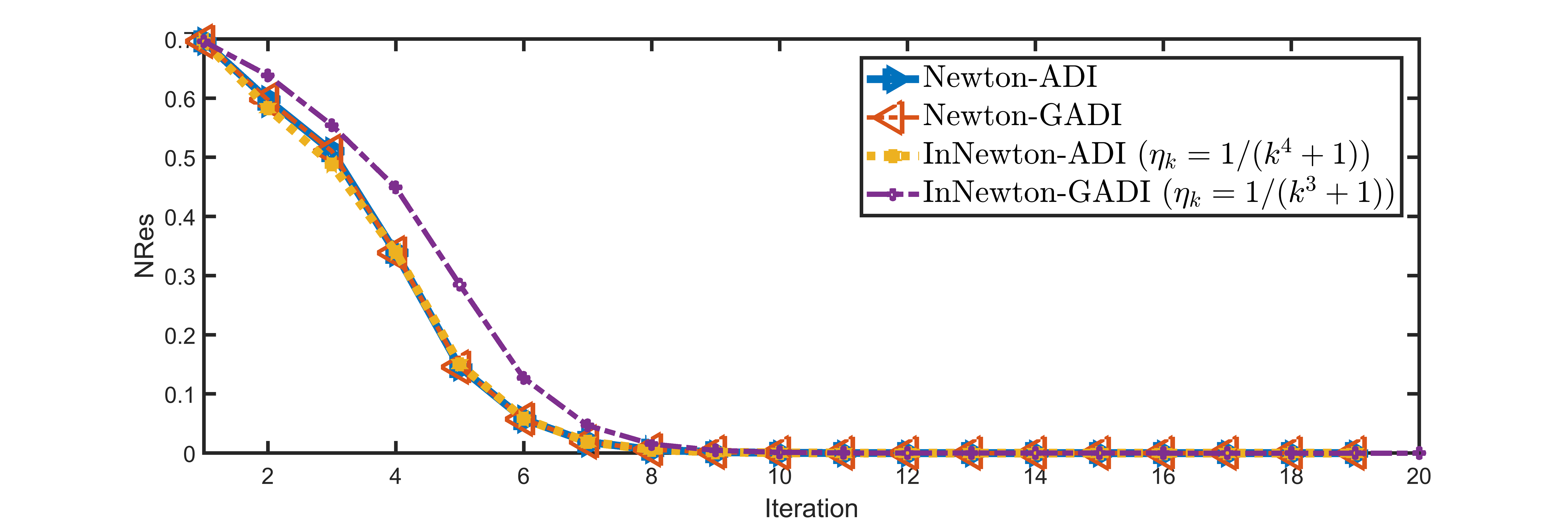}}
\caption{Comparing the convergence rate of the outer iteration for
Newton-ADI, Newton-GADI, InNewton-ADI and InNewton-GADI.}\label{fig1:error2}
\end{figure}

Figure \ref{fig1:IT2} shows that the number of the GADI steps
required for each Newton step. It is easy to see that the inexact Newton-GADI method
has the least number of inter iteration steps among the four algorithms. The total number of steps in the convergence process
is the least compared other methods, see Table \ref{tab:EG1}
for details.

\begin{figure}[!htbp]
\centering
\subfigure{
\includegraphics[width=6in]{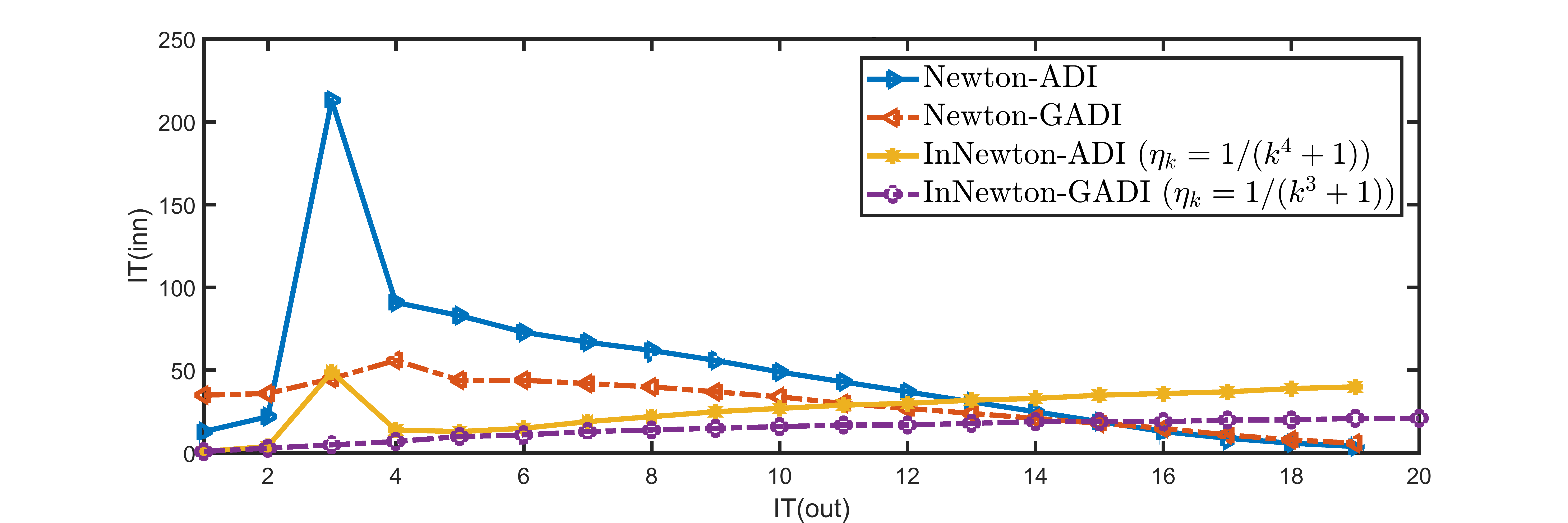}}
\caption{Comparing the number of inner iteration steps of
Newton-ADI, Newton-GADI, InNewton-ADI and InNewton-GADI at each Newton step.}\label{fig1:IT2}
\end{figure}


\begin{table}[!hptb]
\vspace{-0.4cm}
\centering
\caption{Numerical results of Newton-ADI, Newton-GADI, InNewton-ADI and InNewton-GADI.}\label{tab:EG1}
\vspace{0.2cm}
\renewcommand\arraystretch{1.3}
\setlength{\tabcolsep}{3mm}
{\begin{tabular}{c|c|c|c|c|c}\hline
             &$\omega_k^*$  &  $NRes(\tilde{X})$ &IT(out)    & IT(ave)  &IT(cumul)   \\ \hline
Newton-ADI   & 0          &8.1384e-09          &19          &48.21     & 916\\ \hline
Newton-GADI  & 1       & 8.2030e-09            & 19         &30.16     & 573    \\ \hline
InNewton-ADI & 0          & 8.4516e-09         & 26       &26.32       & 500  \\ \hline
InNewton-GADI &1          & 6.6070e-09         & 20       &14.30       & 286  \\ \hline
\end{tabular}}
\end{table}

\begin{example}
Consider the  linear time-invariant system of the form
\begin{align*}
\dot{\bp}(t)=A\bp(t)+B\bu(t)
=\begin{pmatrix}
0 & -1  & 0 & 0\\
1  &0  & -1 & 0\\
0 & 1  &  0  &-1\\
0 & 0  &  1  & 0
\end{pmatrix}\bp(t)
+10^{-3}\times \begin{pmatrix}
    3  & -50 &   1 &   2\\
   1  & -3 &  -2&   1\\
   -3 &  1&  3 &   4\\
   3 & -1&  -4 &  3
\end{pmatrix}\bu(t).
\end{align*}
We wish to minimize the cost function $J$ defined in \eqref{eq:def-cost-J}.
The cost penalty matrices are
\begin{align*}
R = I_4,~
Q= \begin{pmatrix}
    0.0025  &       0 &        0  &       0\\
         0  &  0.0111 &   0.0025  &       0\\
         0  &  0.0025 &   1.0006  &  0.0200\\
         0  &       0 &   0.0200  &  0.0004
\end{pmatrix},~~\mbox{therefore}~~K=BB^{T}.
\end{align*}
\end{example}

It is easy see that $K$ is semi-positive definite and
$Q$ is semi-positive definite. The pair $(A,K)$ is stabilize as
\begin{align*}
\Lambda(A-KI_4)=\{-0.0004 + 1.6180\,i, -0.0004 - 1.6180\,i, -0.0009 + 0.6180\,i,-0.0009 - 0.6180\,i\}.
\end{align*}
The pair $(A,Q)$ is detectable as
\begin{align*}
\Lambda(Q^*-A^*F)=\{
  -0.0013 +50.0608\,i,
  -0.0013 -50.0608\,i,
  -0.0009 + 7.6862\,i,
  -0.0009 - 7.6862i\,\}
\end{align*}
with
\begin{align*}
F=\begin{pmatrix}
   17.5243 &   0.3931 &  -8.2505 &  -0.3925\\
    0.3931 &  25.8550 &   0.3901 &  -8.2707\\
   -8.2505 &   0.3901 &  25.8338 &  -0.0035\\
   -0.3925 &  -8.2707 &  -0.0035 &  17.5505
\end{pmatrix}.
\end{align*}
The numerical Hermitian semi-definite solution is
\begin{align*}
\tilde{X}= \begin{pmatrix}
   17.4818 &   0.3916 &  -8.2438 &  -0.3924\\
    0.3916 &  25.8038 &   0.3901 &  -8.2638\\
   -8.2438 &   0.3901 &  25.7818 &  -0.0035\\
   -0.3924 &  -8.2638 &  -0.0035 &  17.5055
\end{pmatrix}.
\end{align*}

Figure \ref{fign4:errorinn} displays that the error of
the inexact inter iterative algorithm with the increase of Newton step.
Figure \ref{fign4:error} shows the convergence process of the four algorithms and
reports that these algorithms have almost the same convergence rate.
\begin{figure}[!htbp]
\centering
\subfigure{
\includegraphics[width=6in]{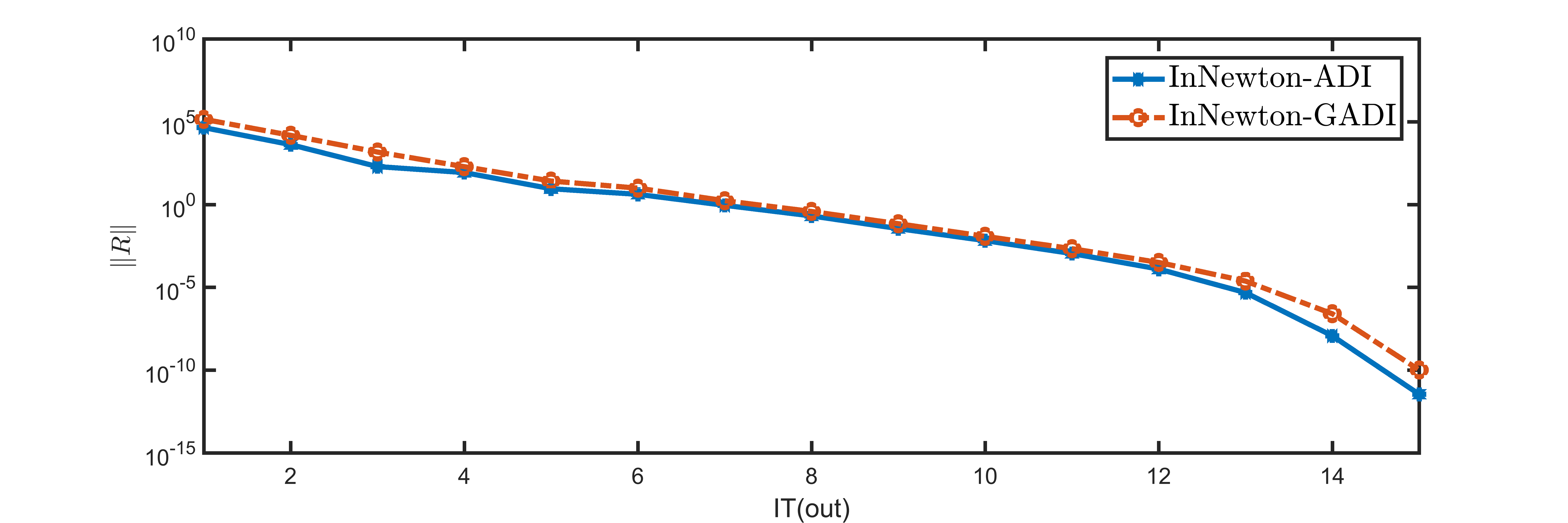}}
\caption{Comparing the inner iteration accuracy of two inexact algorithm with $\eta_k=\frac{1}{k^4+1}$.}\label{fign4:errorinn}
\end{figure}

\begin{figure}[!htbp]
\centering
\subfigure{
\includegraphics[width=6in]{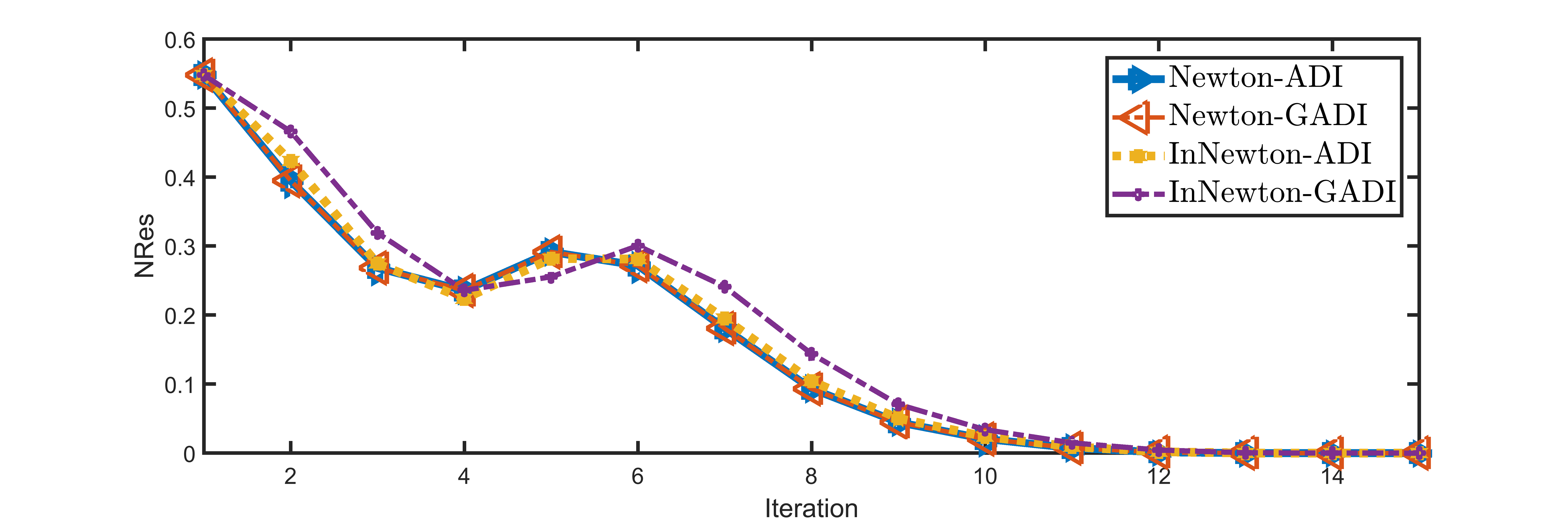}}
\caption{Comparing the convergence rate of the outer iteration for
Newton-ADI, Newton-GADI, InNewton-ADI and InNewton-GADI with $\eta_k=\frac{1}{k^4+1}$ .}\label{fign4:error}
\end{figure}

Figure \ref{fign4:IT} reports the number of the GADI steps
required for each Newton step. It can be seen that the InNewton-GADI algorithm has a clear advantage over other algorithms.

\begin{figure}[!htbp]
\centering
\subfigure{
\includegraphics[width=6in]{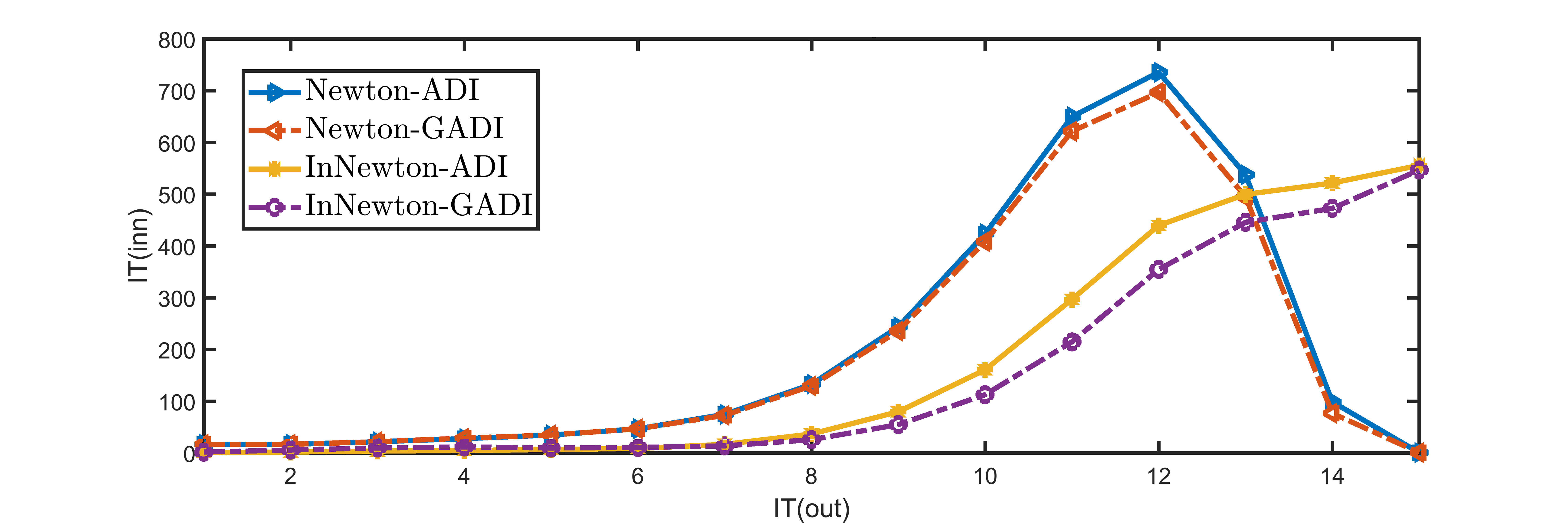}}
\caption{Comparing the number of inner iteration steps of
Newton-ADI, Newton-GADI, InNewton-ADI and InNewton-GADI at each Newton step.}\label{fign4:IT}
\end{figure}

Table \ref{tabn4:EG} displays the numerical results of these four algorithms.
It is not difficult to find that this example is ill-conditioned. Some existing ADI algorithms may be convergent slowly or not convergent, while the GADI scheme proposed in this paper
can accelerate the convergence process.
\begin{table}[!hptb]
\vspace{-0.4cm}
\centering
\caption{Numerical results of Newton-ADI, Newton-GADI, InNewton-ADI and InNewton-GADI.}\label{tabn4:EG}
\vspace{0.2cm}
\renewcommand\arraystretch{1.3}
\setlength{\tabcolsep}{3mm}
{\begin{tabular}{c|c|c|c|c|c}\hline
             &$\omega_k^*$  &  $NRes(\tilde{X})$ &IT(out)    & IT(ave)  &IT(cumul)   \\ \hline
Newton-ADI   & 0          &9.5234e-09          &15          &204.40     &3066\\ \hline
Newton-GADI  & 0.015       &9.5846e-09          &15         &193.80    & 2907    \\ \hline
InNewton-ADI & 0          & 1.2046e-10         & 15       &175.87       & 2638  \\ \hline
InNewton-GADI &0.015       &3.5208e-09         & 15       &153.00       &2295  \\ \hline
\end{tabular}}
\end{table}

\begin{example}
Consider the linear time-invariant system of the form
\begin{align*}
\dot{\bp}(t)=A\bp(t)+B\bu(t)
\end{align*}
where
\begin{align*}
A=\mbox{Tri}(-1-r,-4+8\,i,-1+r)_{n\times n},~
b=(1,0,\cdots,0)\in\bbR^{1\times n},~~B=(b^T, I_n),
\end{align*}
with $r=\dfrac{1}{2n+2}$.
To go further, we need to minimize the cost function $J$ defined in \eqref{eq:def-cost-J}.
The cost penalty matrices are $R=I_n$,~$K=BR^{-1}B^{T}$,~$Q=c^{T}c$ with $c=(\dfrac{1}{\sqrt{10}},0,\cdots,0)\in\bbR^{1\times n}$.
\end{example}

When $n=128$, Figure \ref{fig2:errorinn} displays that the error of
the inexact inter iterative algorithm with the increase of Newton step.
\begin{figure}[!htbp]
\centering
\subfigure{
\includegraphics[width=6in]{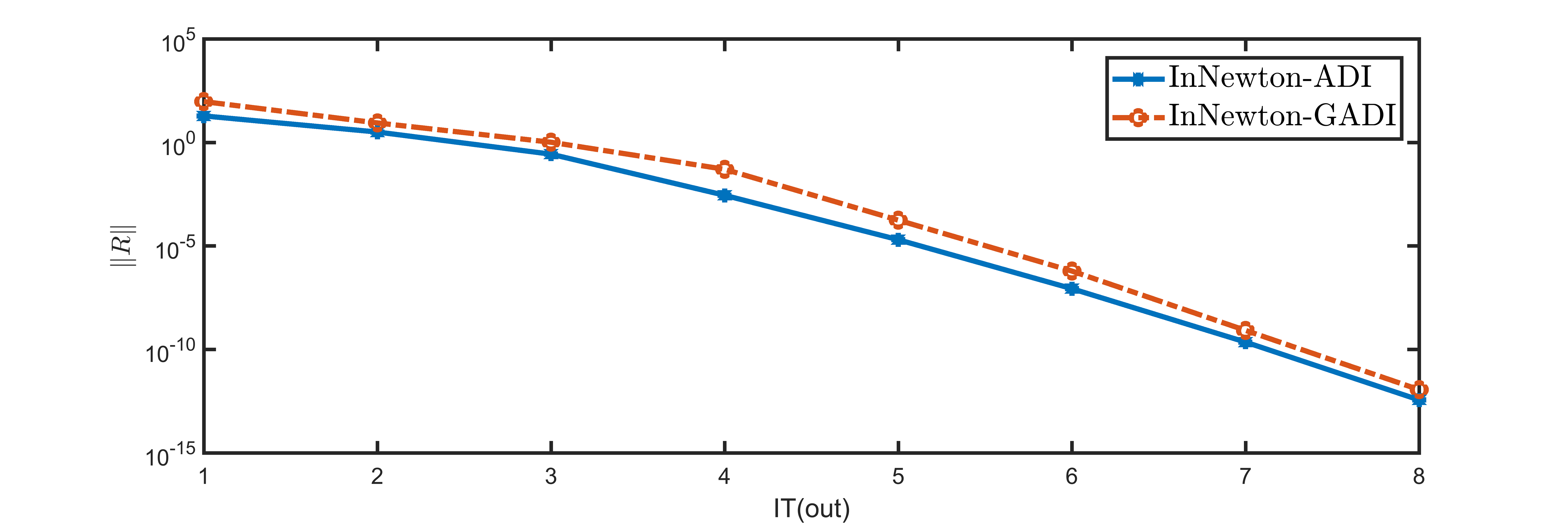}}
\caption{Comparing the inner iteration accuracy of two inexact algorithm with $\eta_k=\frac{1}{k^3+1}$ and $n=128$.}\label{fig2:errorinn}
\end{figure}

Figure \ref{fig2:error} shows the convergence process of the four algorithms and
reports that these algorithms have almost the same convergence rate.
Moreover, it is easy to find that
the convergence rate of the iteration is slow in the beginning.
From the fifth iteration step, these algorithms begin to converge rapidly.
\begin{figure}[!htbp]
\centering
\subfigure{
\includegraphics[width=6in]{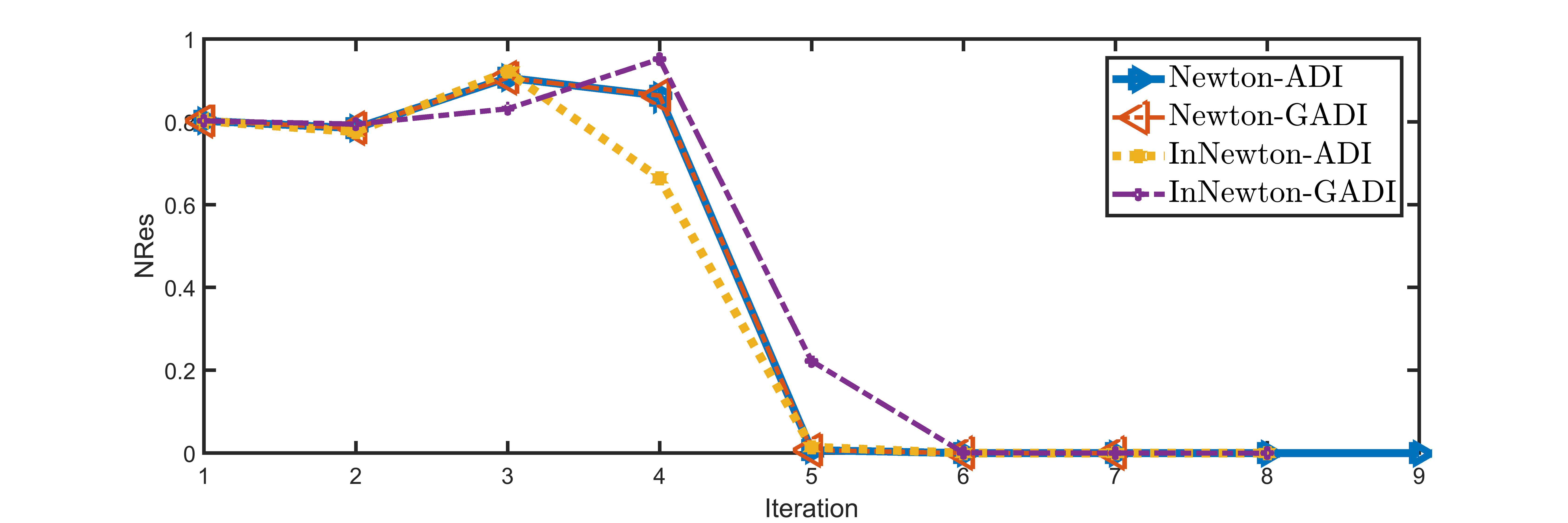}}
\caption{When $n=128$, comparing the convergence rate of the outer iteration for
Newton-ADI, Newton-GADI, InNewton-ADI and InNewton-GADI with $\eta_k=\frac{1}{k^3+1}$ .}\label{fig2:error}
\end{figure}

Figure \ref{fig2:IT} reports the number of the GADI steps
required for each Newton step. It can be seen that the InNewton-GADI algorithm has a clear advantage over other algorithms.

\begin{figure}[!htbp]
\centering
\subfigure{
\includegraphics[width=6in]{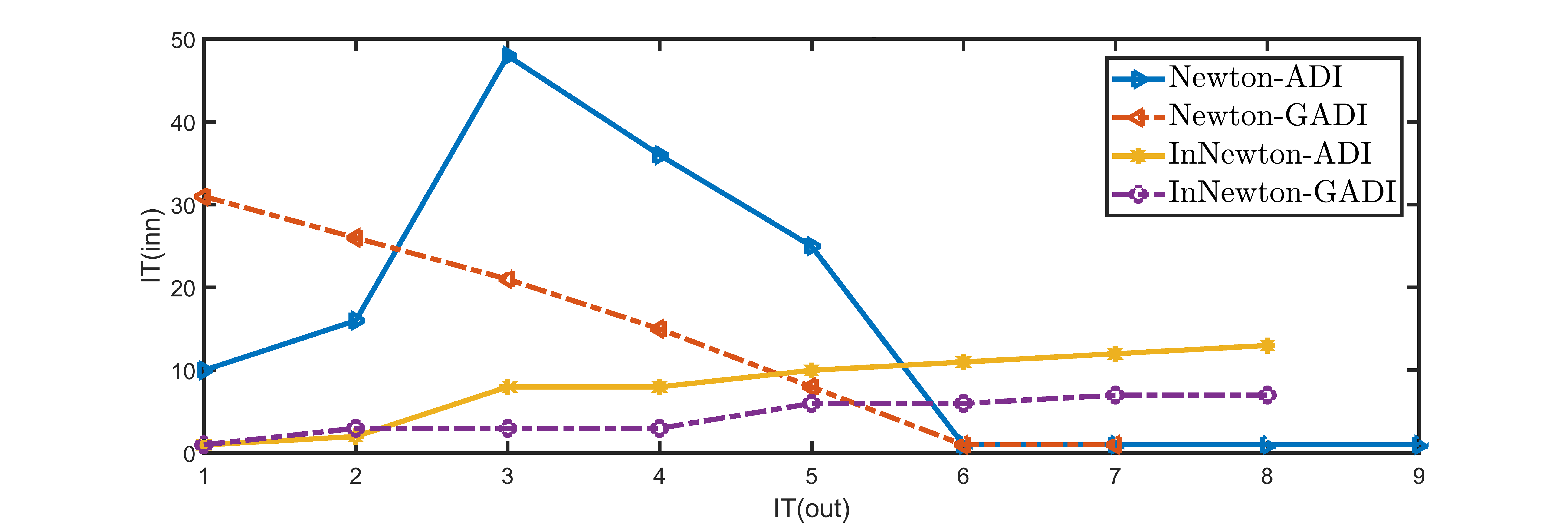}}
\caption{When $n=128$, comparing the number of inner iteration steps of
Newton-ADI, Newton-GADI, InNewton-ADI and InNewton-GADI at each Newton step.}\label{fig2:IT}
\end{figure}

Table \ref{tab:EG2} displays the numerical results of these four algorithms
when $n=64,~n=128,~n=512$ and $n=1024$.
We can see that as $n$ increases, the number of iterative steps and the finally normalized residual $NRes(\tilde{X})$ of each algorithm are almost the same,
while the CPU time increases gradually.
\begin{table}[!hptb]
\vspace{-0.4cm}
	\centering
	\caption{Numerical results of Newton-ADI, Newton-GADI, InNewton-ADI and InNewton-GADI
as $n$ increases.}\label{tab:EG2}
    \vspace{0.2cm}
	\renewcommand\arraystretch{1.3}
	\setlength{\tabcolsep}{3mm}
	{\begin{tabular}{c|c|c|c|c|c|c|c}\hline
$n$               &  Algorithm  &$\omega_k^*$  & $NRes(\tilde{X})$                &IT(out)    & IT(ave)  &IT(cumul) & CPU(s)  \\ \hline
\multirow{4}*{64} &Newton-ADI & 0          &6.3978e-09          &11    &12.73     & 140  & 0.2140 \\ \cline{2-8}
                &Newton-GADI  & 1       & 9.4256e-09            & 7     &14.71     &103  & 0.1320 \\ \cline{2-8}
                &InNewton-ADI & 0          &1.1340e-09         &8       &8.13       &65  & 0.1070\\ \cline{2-8}
                 &InNewton-GADI &1          &4.2259e-09        &8       &4.50       &36  & 0.0830\\ \hline
\multirow{4}*{128}  &Newton-ADI   & 0     &8.4557e-09          &9        &15.44     &139 &0.7540  \\ \cline{2-8}
                    &Newton-GADI  & 1       & 9.4644e-09        & 7    &14.72       &103 &0.5400   \\ \cline{2-8}
                   &InNewton-ADI & 0          &1.1531e-09         & 8    &8.13      &65  &0.5150  \\ \cline{2-8}
                   &InNewton-GADI &1          &4.2424e-09         & 8     &4.50     &36  &0.3680  \\ \hline
 \multirow{4}*{512} &Newton-ADI   & 0     &8.4857e-09          &9      &15.44     &139  &22.1390 \\ \cline{2-8}
                    &Newton-GADI  & 1       & 9.4771e-09       & 7   &14.72     & 103  &16.9060    \\ \cline{2-8}
                   &InNewton-ADI & 0          &1.1596e-09      & 8     &8.13    &65    &15.5520 \\ \cline{2-8}
                  &InNewton-GADI &1      & 4.2480e-09         & 8      &4.50    & 36   &10.8330 \\ \hline
\multirow{4}*{1024} &Newton-ADI   & 0     &8.4872e-09          &9      &15.44     &139  &256.4320 \\ \cline{2-8}
                    &Newton-GADI  & 1       & 9.4777e-09       & 7   &14.72     & 103  &201.9130    \\ \cline{2-8}
                   &InNewton-ADI & 0          &1.1600e-09      & 8     &8.13    &65  &159.0060 \\ \cline{2-8}
                  &InNewton-GADI &1      & 4.2484e-09         & 8      &4.50    & 36 &113.2540 \\ \hline
\end{tabular}}
\end{table}

\section{Conclusion}

In this paper, we have proposed the Newton-GADI algorithm to solve the complex CARE. Moreover, by inaccurately solving a Lyapunov equation in each Newton step,
 the inexact Newton-GADI algorithm is presented.
Finally, we have offered numerical experiments to compare the effectiveness.
It can be seen that the Newton-GADI method has a faster inter iterative convergence rate
than the Newton-ADI method, while the inexact Newton-GADI method is more efficient.
Therefore, the Newton-GADI and inexact Newton-GADI methods could be valid and attractive algorithms
to solve the ill-conditioned Riccati equation.
In the future, we will further study parameter selection and solving the low-rank large-scale CARE.

\nocite{*}

\end{document}